\documentclass[mathpazo]{cicp}
\usepackage[margin=1in]{geometry}
\usepackage{amsmath,epsf,cite}
\usepackage{amssymb,amsthm,tocvsec2}
\usepackage{graphicx}
\usepackage{epsfig,epsf,latexsym,subfigure}
\usepackage{float}
\usepackage{color}
\usepackage{mathrsfs}
\usepackage{indentfirst}

\usepackage[width=.75\textwidth]{caption}

\date{ }

\numberwithin{equation}{section}
\numberwithin{figure}{section}
\numberwithin{table}{section}

\theoremstyle{plain}

\theoremstyle{remark}
\newtheorem{rem}{Remark}[section]

\def\bx{\mathbf{x}}

\def\bG{\mathbf{G}}

\def\cL{\mathcal{L}}

\def\cG{\mathcal{G}}
\def\cB{\mathcal{B}}

\def\cN{\mathcal{N}}
\def\cE{\mathcal{E}}

\newcommand{\ben}{\begin{eqnarray}}
\newcommand{\een}{\end{eqnarray}}
\newcommand{\beq}{\begin{equation}}
\newcommand{\eeq}{\end{equation}}
\newcommand{\bea}{\begin{array}}
\newcommand{\eea}{\end{array}}
\newcommand{\bef}{\begin{figure}[H]}
\newcommand{\eef}{\end{figure}}

\newtheorem{scheme}{Scheme}[section]

\begin{document}
\title{A Revisit of The Energy Quadratization Method with A Relaxation Technique}
\author[J. Zhao]{
Jia Zhao\affil{1}\comma \corrauth}
\address{\affilnum{1}\ Department of Mathematics \& Statistics, Utah State University, Logan, UT, USA }
\email{ {\tt jia.zhao@usu.edu.} (J.~Zhao)}

\begin{abstract}
This letter revisits the energy quadratization (EQ) method by introducing a novel and essential relaxation technique to improve its accuracy and stability. The EQ method has witnessed significant popularity in the past few years. Though acknowledging its effectiveness in designing energy-stable schemes for thermodynamically consistent models, the primary known drawback is apparent, i.e., its preserves a "modified" energy law represented by auxiliary variables instead of the original variables.  Truncation errors are introduced during numerical calculations so that the numerical solutions of the auxiliary variables are no longer equivalent to their original continuous definitions. Even though the "modified" energy dissipation law is preserved, the original energy dissipation law is not guaranteed. In this paper, we overcome this issue by introducing a relaxation technique. The computational cost of this extra technique is negligible compared with the baseline EQ method. Meanwhile, the relaxed-EQ method holds all the baseline EQ method's good properties, such as linearity and unconditionally energy stability. Then we apply the relaxed-EQ method to several widely-used phase field models to highlight its effectiveness.
\end{abstract}

\maketitle

\section{Background}
In this letter, we focus on the PDE models that respect the thermodynamical laws. They are usually named thermodynamically consistent models. In particular, when the temperature deviation is considered, the PDE model's entropy shall be non-decreasing in time. When the temperature deviation is ignored, the PDE model's Helmholtz free energy shall be non-increasing in time. These thermodynamically consistent models are broadly used to investigate problems in various fields, taking material science, life science, and mechanical engineering as examples.

In the past few decades, a community is shaped on developing numerical algorithms to solve the thermodynamically consistent models while preserving their thermodynamical structures \cite{Wang&Wang&WiseDCDS2010,Qiao&Zhang&TangSISC2011,Guill2013On,SAV-1,Han&WangJCP2015}. Among many seminal publications,  the energy quadratization (EQ) method has witnessed an intense exploration \cite{Yang&Zhao&WangJCP2017,Zhao&Yang&Gong&Zhao&Yang&Li&Wang2018,GongEnergy}. 

To illustrate the idea of the EQ method, we start with a simple example, the Allen-Cahn (AC) equation with periodic boundary condition, given as
\beq \label{eq:AC}
\partial_t \phi(\bx, t) = \varepsilon^2 \Delta \phi(\bx, t)  -\phi(\bx, t)^3 + \phi(\bx, t), \quad (\bx, t) \in  \Omega \times (0, T].
\eeq 
Define the free energy 
$$F = \int_\Omega \frac{\varepsilon^2}{2}|\nabla \phi|^2 + \frac{1}{4}(\phi^2-1)^2 d\bx,$$ 
with $\varepsilon$ a model parameter. It is known that the AC equation in \eqref{eq:AC} satisfies the energy dissipation law
\beq \label{eq:AC-energy-law}
\frac{dF}{dt} = \int_\Omega \frac{\delta F}{\delta \phi} \frac{\partial \phi}{\partial t} d\bx = - \int_\Omega (\frac{\partial \phi}{\partial t})^2 d\bx \leq 0.
\eeq 
To design numerical schemes that preserve the energy dissipation law for the AC equation, the EQ method introduces a novel perspective, namely, instead of solving the original AC equation in \eqref{eq:AC}, it solves an equivalent problem.  By introducing the auxiliary variable
$$
q(\bx, t) :=h(\phi)= \frac{\sqrt{2}}{2}(\phi^2 -1),
$$ 
and the notation $g(\phi):=\frac{\partial q}{\partial \phi} = \sqrt{2} \phi$,
the equivalent problem is formulated as
\begin{subequations} \label{eq:AC-EQ}
\begin{align}
&\partial_t \phi = \varepsilon^2 \Delta \phi(\bx, t) - q g(\phi), \\
&\partial_t q  = g(\phi) \partial_t \phi,  
\end{align}
\end{subequations}
with consistent initial condition $q|_{t=0} =  \frac{\sqrt{2}}{2}(\phi^2|_{t=0} -1)$. The reformulated problem has an equivalent energy dissipation law
\beq \label{eq:AC-EQ-energy-law}
\frac{dE}{dt} = \int_\Omega \frac{\delta E}{\delta \phi} \frac{\partial \phi}{\partial t} + \frac{\delta E}{\delta q} \frac{\partial q}{\partial t} d\bx =  - \int_\Omega (\frac{\partial \phi}{\partial t})^2 d\bx \leq 0,
\eeq 
where the modified free energy is defined as
$
E=\int_\Omega \frac{\varepsilon^2}{2}|\nabla \phi|^2 + \frac{1}{2} q^2 d\bx.
$
In the PDE level, it is known that $q=h(\phi)=\frac{\sqrt{2}}{2}(\phi^2-1)$, so that the reformulated problem in \eqref{eq:AC-EQ} is equivalent with \eqref{eq:AC}, as well as their energy dissipation laws.

However,  the numerical solutions for $q$ and the numerical results of $h(\phi)$ are not necessarily equal any more (in fact, they are not) after temporal discretization. Hence a discrete energy dissipation law in \eqref{eq:AC-EQ-energy-law} is no longer equivalent with a discrete energy dissipation law in \eqref{eq:AC-energy-law}, though we shall recognize that they are consistent with the numerical schemes' order of temporal accuracy.

With this in mind, the primary goal of this letter is to overcome this inconsistency issue. And our central idea is to introduce a relaxation technique for the numerical solutions of $q$, to penalize the numerical deviation from the original definition of $q:=h(\phi)$.

\section{Thermodynamically consistent reversible-irreversible PDE models}
\subsection{A generic model formulation}
Consider a domain $\Omega$, and denote the thermodynamic variable as $\phi$. We introduce the general formulation for a thermodynamically consistent reversible-irreversible PDE models as
\begin{subequations} \label{eq:evolution-general}
\begin{align}
&  \partial_t \phi (\bx,t) = - \cG \frac{\delta \cE}{\delta \phi} \mbox{ in } \Omega, \\
&\cB(\phi(\bx,t)) = g(\bx,t), \mbox{ on } \partial \Omega,
\end{align}
\end{subequations}
where  $\cB$ is a trace operator,  and $\cG:=\cG_a+\cG_s$ is the mobility operator,
$\cG_s$  is symmetric and positive semi-definite  to ensure thermodynamically consistency,  $\cG_a$  is  skew-symmetric,  and $\frac{\delta \cE}{\delta \phi}$ is the variational derivative of $\cE$, known as the chemical potential. Then, the triplet $(\phi,\cG,\cE)$ uniquely defines a thermodynamically consistent model. One intrinsic property  of \eqref{eq:evolution-general} owing to the thermodynamical consistency  is the energy dissipation law
\begin{subequations} \label{EDL}
\begin{align}
& \frac{d \cE}{d t} = \Big( \frac{\delta \cE}{\delta \phi},   \frac{\partial\phi}{\partial t} \Big)+ \dot{\cE}_{surf} =\dot{\cE}_{bulk}+\dot{\cE}_{surf},\\
&\dot{\cE}_{bulk}=-\Big( \frac{\delta \cE}{\delta \phi}, \cG_s \frac{\delta F}{\delta \phi} \Big) \leq 0, \quad \Big( \frac{\delta \cE}{\delta \phi}, \cG_a \frac{\delta \cE}{\delta \phi} \Big) = 0, \quad  \dot{\cE}_{surf}=\int_{\partial \Omega} g_b ds, 
\end{align}
\end{subequations}
where the inner product is defined by
$(f, g) = \int_\Omega f g d\bx$,$\forall f, g \in L^2(\Omega)$,
and $\dot{\cE}_{surf}$ is due to the boundary contribution, and $g_b$ is the boundary integrand.
When $\cG_a=0$, \eqref{eq:evolution-general} is a purely dissipative system; while $\cG_s=0$, it is a purely dispersive system. $\dot{\cE}_{surf}$ vanishes only for suitable boundary conditions, which include periodic and certain physical boundary conditions.

\subsection{Model reformulation with the energy quadratization method}
Now,  we illustrate the energy quadratization (EQ) idea on solving \eqref{eq:evolution-general}. Denote the total energy as
$\cE(\phi) = \int_\Omega e d\bx$, 
with $e$ the energy density function. We denote $L_0$ as a linear operator that can be separated from $e$. 
Introduce the auxiliary variable
\beq \label{eq:EQ-intermediate-variable}
q =: h(\phi) = \sqrt{2\Big(e -\frac{1}{2}|\cL_0^{\frac{1}{2}} \phi|^2+ \frac{A_0}{|\Omega|}\Big)},
\eeq
where  $A_0>0$ such that $q$ is a well defined real variable. Then we rewrite the energy as
\beq
\cE(\phi, q) =  \frac{1}{2}\Big(\phi, \cL_0\phi \Big) + \frac{1}{2}\Big( q, q \Big) - A_0,
\eeq 
With the EQ approach above,  we transform the free energy density into a quadratic one by introducing an auxiliary variable to "remove" the quadratic gradient term from the energy density.
Assuming $q =q(\phi,\nabla \phi)$ and denoting
$g(\phi)= \frac{\partial q}{\partial \phi}$ and  $\bG(\nabla\phi) = \frac{\partial q}{\partial \nabla \phi}$,
we reformulate  \eqref{eq:evolution-general} into an equivalent form
\begin{subequations} \label{eq:evolution-general-EQ}
\begin{align}
&\partial_t \phi = -(\cG_a + \cG_s) \Big[ \cL_0 \phi + q g(\phi) - \nabla \cdot ( q \bG(\nabla\phi))\Big] , \\
&\partial_t q = g(\phi): \partial_t \phi + \bG(\nabla \phi): \nabla \partial_t \phi, 
\end{align}
\end{subequations}
with the consistent initial condition $q|_{t=0} =\left. \sqrt{2\Big(e -\frac{1}{2}|\cL_0^{\frac{1}{2}} \phi|^2+ \frac{A_0}{|\Omega|}\Big)}\right|_{t=0}$.
Now, instead of dealing with \eqref{eq:evolution-general} directly, we develop   structure-preserving schemes for    \eqref{eq:evolution-general-EQ}.

The advantage of using  model \eqref{eq:evolution-general-EQ} over   model \eqref{eq:evolution-general} is that the   energy density is transformed into a quadratic one  in  \eqref{eq:evolution-general-EQ}.
Denoting $\Psi =\begin{bmatrix} \phi \\ q \end{bmatrix}$, we rewrite \eqref{eq:evolution-general-EQ} into a compact from
\beq \label{eq:evolution-general-EQ-vector}
\partial_t \Psi  = -\cN(\Psi) \cL \Psi, 
\eeq 
where  $\cL$ is a linear operator, and $\cN(\Psi)$ is the mobility operator as
\beq 
\cN(\Psi) = \cN_s(\Psi) + \cN_a(\Psi),  \quad \cN_a(\Psi) = \cN_0^\ast \cG_a  \cN_0 , \quad  \cN_s(\Psi) = \cN_0^\ast \cG_s \cN_0.
\eeq
Here
$\cN_0= (1, \,\,\, g(\phi) + \bG(\nabla \phi)\colon\nabla)$,
$\cL = \begin{bmatrix}
\cL_0 & \\
& 1
\end{bmatrix},
$
and $\cN_0^\ast$ is the adjoint operator of $\cN_0$. When $\dot{\cE}_{surf}=0$, the energy law is 
\beq  \label{eq:energy-law-EQ}
\frac{d \cE(\Psi)}{d t}
= \Big( \frac{\delta \cE}{\delta \Psi} \frac{d \Psi}{d t}, 1 \Big)
= -\Big( \ \cL \Psi,  \cN(\Psi) \cL \Psi  \Big) = -\Big( \cN_0 \cL \Psi,  \cG_s \cN_0 \cL \Psi  \Big)   \leq  0.
\eeq

\subsection{Model examples}
To demonstrate this methodology's generality, we introduce a few widely-used PDE models that can be cast as special cases.

\textbf{Allen Cahn equation.}
First of all, let us consider the Allen-Cahn equation
\beq
\partial_t \phi = \varepsilon^2 \Delta \phi - \phi^3 +\phi.
\eeq 
It can be formulated in an energy variation form of \eqref{eq:evolution-general} with $\cG=1$ and $F = \int_\Omega \frac{\varepsilon^2}{2}|\nabla \phi|^2 + \frac{1}{4}(\phi^2-1)^2d\bx$.
Next, if we introduce $q := h(\phi) = \frac{1}{\sqrt{2}}(\phi^2 -1 - \gamma_0)$, we will have $g(\phi) :=\frac{\partial q}{\partial \phi} = \sqrt{2}\phi$. Then the model is reformulated into the quadratization form of \eqref{eq:evolution-general-EQ} as
\begin{subequations}
\begin{align}
& \partial_t \phi = \varepsilon^2 \Delta \phi - \gamma_0 \phi - g(\phi) q, \\
& \partial_t q = g(\phi) \partial_t \phi,
\end{align}
\end{subequations}
with the consistent initial condition $q|_{t=0} = h(\phi|_{t=0})$.
Its matrix formulation of \eqref{eq:evolution-general-EQ-vector} can also be easily obtained, by noticing $\cL_0 =-\varepsilon^2 \Delta + \gamma_0$.  

\textbf{Cahn-Hilliard (CH) equation.}
In the second example, we consider the Cahn-Hilliard equation given as
\begin{subequations}
\begin{align}
& \partial_t \phi = M \Delta \mu, \\
& \mu = -\varepsilon^2 \Delta \phi + f'(\phi),
\end{align}
\end{subequations}
where $f(\phi)$ is the bulk potential. In this paper, we choose the double-well potential $f(\phi) = \frac{1}{4}(\phi^2-1)^2$. Other potentials, such as the Flory-Huggins potential can also be used. It can be reformulated in an energy variational form of \eqref{eq:evolution-general} with the notations  $\cG = -\Delta$, and $F =\int_\Omega \Big[  \frac{\varepsilon^2}{2}|\nabla \phi|^2 + f(\phi) \Big] d\bx$.
Then, we introduce $q:=h(\phi) = \frac{1}{\sqrt{2}}(\phi^2-1-\gamma_0)$ and $g(\phi):=\frac{\partial q}{\partial \phi} = \sqrt{2} \phi$, leading us to the quadratized form of \eqref{eq:evolution-general-EQ} as
\begin{subequations}
\begin{align}
&\partial_t \phi = \Delta (-\varepsilon^2 \phi + \gamma_0 \phi - g(\phi) q), \\
&\partial_t q = g(\phi) \partial_t \phi,
\end{align}
\end{subequations}
with the consistent initial condition $q|_{t=0} = h(\phi|_{t=0})$.

\textbf{Molecular beam epitaxy (MBE) model.}
Then, we focus on the molecular beam epitaxy (MBE) model with slope section. The model reads
\beq
\partial_t \phi = -M \Big[ \varepsilon^2 \Delta \phi^2 - \nabla \cdot ( (|\nabla \phi|^2  -1) \nabla \phi) \Big].
\eeq 
Similarly, this model can be written in an energy variation form of \eqref{eq:evolution-general} as $\cG=M$ and $F = \int_\Omega \Big[ \frac{\varepsilon^2}{2}(\Delta \phi)^2 + \frac{1}{4}(|\nabla \phi|^2 - 1)^2 \Big] d\bx$.
If we introduce the notations
$q :=h(\phi) =\frac{\sqrt{2}}{2}(|\nabla \phi|^2 - 1 -\gamma_0)$, and $\bG(\nabla \phi) :=\frac{\partial q}{ \partial \nabla \phi} = \sqrt{2} \nabla \phi
$, the MBE model can be written in the quadratized form as
\begin{subequations}
\begin{align}
& \partial_t \phi = -M \Big[ \varepsilon^2 \Delta^2 \phi - \gamma \Delta \phi  - \nabla \cdot (q \bG(\nabla \phi)) \Big], \\
& \partial_t q = \bG(\nabla \phi) : \nabla \partial_t \phi,
\end{align}
\end{subequations} 
with a consistent initial condition for $q$.

\textbf{Phase field crystal (PFC) equation.}
Next, we investigate the phase field crystal (PFC) equation which reads as
\beq
\partial_t \phi = \Delta ( (a_0+\Delta )^2 \phi + f'(\phi)), \quad f(\phi) =\frac{1}{4} \phi^4 - \frac{b_0}{2} \phi^2,
\eeq 
with $a_0$ and $b_0$ are model parameters. The PFC model can also be written in an energy variational form of \eqref{eq:evolution-general} with $\cG= \Delta$ , and $F = \int_\Omega  \Big\{  \frac{1}{2} \phi \Big[ -b_0 + (a_0+\Delta)^2  \Big] \phi + \frac{1}{4}\phi^4 \Big\} d\bx$. 
Then, if we introduce the notations 
$q :=h(\phi)= \sqrt{2} \sqrt{\frac{1}{4}\phi^4 - \frac{b_0}{2} \phi^2 - \frac{\gamma_0}{2}\phi^2 + C_0}$ and $ g(\phi): = \frac{\partial q}{\partial \phi} =\frac{\phi^3- (b_0+\gamma_0) \phi }{\sqrt{2} \sqrt{\frac{1}{4}\phi^4 - \frac{b_0}{2} \phi^2 - \frac{\gamma_0}{2}\phi^2 + C_0}}$,
we can rewrite the PFC model into the quadratizated form of \eqref{eq:evolution-general-EQ} as
\begin{subequations}
\begin{align}
&\partial_t \phi = \Delta \Big[ (a_0+\Delta)^2 \phi + \gamma_0\phi + g(\phi) q \Big], \\
&\partial_t q = g(\phi) \partial_t \phi.
\end{align}
\end{subequations}

\section{The relaxed energy quadratization method}

\subsection{Baseline EQ schemes}
We recall the classical EQ numerical algorithms to solve the generalized thermodynamically consistent problems. Mainly, we discretize all the linear terms implicitly and all the nonlinear terms explicitly. We will end up with some energy stable and linear numerical algorithms.
\begin{scheme}[CN Scheme] \label{scheme:CN}
After we calculate $\Psi^{n-1}$ and $\Psi^{n}$, we can update $\Psi^{n+1}$ via the following scheme
$$
\frac{\Psi^{n+1} - \Psi^n}{\delta t} = - \cN(\overline{\Psi}^{n+\frac{1}{2}}) \cL\Psi^{n+\frac{1}{2}},
$$
with the notation $\overline{\Psi}^{n+\frac{1}{2}} = \frac{3}{2}\Psi^n -\frac{1}{2}\Psi^{n-1}$. 
\end{scheme}
\begin{scheme}[BDF2 Scheme] \label{scheme:BDF2}
After we calculate $\Psi^{n-1}$ and $\Psi^n$, we can update $\Psi^{n+1}$ via the following scheme
$$
\frac{3\Psi^{n+1} - 4\Psi^n+ \Psi^{n-1}}{2\delta t} = - \cN(\overline{\Psi}^{n+1}) \cL\Psi^{n+1}.
$$
with the notation $\overline{\Psi}^{n+1} = 2\Psi^n - \Psi^{n-1}$
\end{scheme}

It could be easily verified that the EQ schemes are uniquely solvable and unconditionally energy stable \cite{Zhao&Yang&Gong&Zhao&Yang&Li&Wang2018}. Meanwhile, the known issue is that energy is modified, since $q$ does not necessarily equal to its original definition of \eqref{eq:EQ-intermediate-variable} after discretization, i.e. $q^{n+1} \neq h(\phi^{n+1})$ in general. 

\subsection{Relaxed EQ schemes}
To overcome this inconsistency issue of the modified energy and the original energy in the baseline EQ method above, we come up with the novel relaxed EQ schemes as follows, as an analogy to the two baseline EQ schemes.

\begin{scheme}[Relaxed CN Scheme] \label{scheme:CN-relax}
After we calculate $\Psi^{n-1}$ and $\Psi^{n}$, we can update $\Psi^{n+1}$ via the following two step schemes:
\begin{itemize}
\item Step 1, use the baseline EQ method to obtain the intermediate solutions. Denote $\hat{\Psi}^{n+\frac{1}{2}} = \frac{1}{2} \hat{\Psi}^{n+1} + \frac{1}{2} \Psi^n$, and we find $\hat{\Psi}^{n+1} := \begin{bmatrix}
\phi^{n+1} \\
\hat{q}^{n+1}
\end{bmatrix}$, via solving the equation
\beq  \label{eq:CN-step1}
\frac{\hat{\Psi}^{n+1} - \Psi^n}{\delta t} = - \cN(\overline{\Psi}^{n+\frac{1}{2}}) \cL \hat{\Psi}^{n+\frac{1}{2}}.
\eeq 

\item Step 2, update the baseline solutions with a relaxation step. We update $q^{n+1}$ as 
\beq \label{eq:CN-xi}
q^{n+1} = \xi_0 \hat{q}^{n+1} + (1-\xi_0) h(\phi^{n+1}),  
\eeq 
where $\xi_0$ is calculated from
$$\xi_0 = \max\{0, \frac{-b - \sqrt{b^2 - 4ac}}{2a}\},$$
with the coefficient $a$, $b$ and $c$ given by
\begin{subequations}
	\begin{align}
	& a = \frac{1}{2}  \| \hat{q}^{n+1} - h(\phi^{n+1}) \|^2, \quad b = \Big(\hat{q}^{n+1}, h(\phi^{n+1}) \Big) - \| h(\phi^{n+1}) \|^2,\\
	& c = \frac{1}{2} \| h(\phi^{n+1})\|^2 - \frac{1}{2}\| \hat{q}^{n+1} \|^2 - \delta t  \eta \Big( \cL\hat{\Psi}^{n+\frac{1}{2}}, \cN(\overline{\Psi}^{n+\frac{1}{2}}) \cL \hat{\Psi}^{n+\frac{1}{2}}  \Big).
	\end{align}
\end{subequations}
\end{itemize}
Here $\eta \in [0, 1]$ is an artificial parameter that can be assigned.
\end{scheme}

\begin{rem}
Here we explain the idea for the relaxation step. $\xi_0$ is a solution for the optimization problem: 
\beq  
\xi_0 = \min_{\xi \in [0, 1]} \xi, \quad  \mbox{ s.t. }
\frac{1}{2}(q^{n+1} , q^{n+1}) - \frac{1}{2}( \hat{q}^{n+1}, \hat{q}^{n+1}) \leq \delta t  \eta \Big( \cL\hat{\Psi}^{n+\frac{1}{2}}, \cN(\overline{\Psi}^{n+\frac{1}{2}}) \cL \hat{\Psi}^{n+\frac{1}{2}}  \Big).
\eeq 
It is not hard to show that $\xi = 1$ is in the feasible domain, since $a+b+c \leq 0$. In addition with  $a>0$, we can easily have $\Delta = b^2 - 4ac \geq 0$. Then we have two real roots for the quadratic equation $a\xi^2 + b\xi + c=0$ that are given by $\xi = \frac{-b \pm \sqrt{b^2 - 4ac}}{2a}$.  Since $ \delta t \eta  \Big(\cL \hat{\Psi}^{n+\frac{1}{2}}, \cN(\overline{\Psi}^{n+\frac{1}{2}}) \cL \hat{\Psi}^{n+\frac{1}{2}}  \Big) \geq 0$, we know that $ \frac{-b - \sqrt{b^2 - 4ac}}{2a} \leq 1$, so that we obtain the solution for \eqref{eq:CN-xi} as
$$\xi_0 = \max\{0, \frac{-b - \sqrt{b^2 - 4ac}}{2a}\}.$$
\end{rem}

\begin{rem}
We emphasize that $\eta \in [0, 1]$  in \eqref{eq:CN-xi} controls the relaxation strength.
When $\eta=0$, a minimum relaxation is introduced. When $\eta=1$ a maximum relaxation is introduced. 
\end{rem}

\begin{rem}
The Scheme \ref{scheme:CN-relax} is second-order accurate in time. Notice the fact, $\hat{q}^{n+1} = q(t_{n+1}) + O(\delta t^2)$ and $h(\phi^{n+1}) = q(t_{n+1}) +O(\delta t^2)$. Thus, $\forall \xi \in [0, 1]$, $\xi \hat{q}^{n+1} + (1-\xi) h(\phi^{n+1}) = q(t_{n+1}) + O(\delta t^2)$. I.e., the relaxation step doesn't affect the order of accuracy for the proposed schemes.
\end{rem}


\begin{rem}
We point out the following connections. When $\xi_0=1$, the Scheme \ref{scheme:CN-relax} reduces to the baseline CN scheme \ref{scheme:CN}. When $\xi_0 \in (0, 1)$, the Scheme \ref{scheme:CN-relax} relaxes the solution $q^{n+1}$ to be closer to its original definition of \eqref{eq:EQ-intermediate-variable} in the continuous level. And when $\xi_0=0$, the solution $q^{n+1}$ is consistent to its original definition of \eqref{eq:EQ-intermediate-variable}. In practice, the artificial parameter $\eta \in [0, 1]$ is used to control how much relaxation is added.
\end{rem}

\begin{theorem}
The Scheme \ref{scheme:CN-relax} is unconditionally energy stable.
\end{theorem}

\begin{proof}
As a matter of fact, if we take inner product of \eqref{eq:CN-step1} with $\delta t \cL \hat{\Psi}^{n+\frac{1}{2}}$, we will have
\beq \label{eq:proof-cn-step1}
\frac{1}{2} \Big[ (\phi^{n+1}, \cL_0 \phi^{n+1}) + \|\hat{q}^{n+1}\|^2 \Big] - \frac{1}{2}  \Big[ (\phi^n, \cL_0 \phi^n) + \| q^n\|^2 \Big] = -\delta t \Big(\cL \hat{\Psi}^{n+\frac{1}{2}}, \cN(\overline{\Psi}^{n+\frac{1}{2}}) \cL \hat{\Psi}^{n+\frac{1}{2}}  \Big).
\eeq 
Also, from \eqref{eq:CN-xi}, we know
\beq \label{eq:proof-cn-step2}
\frac{1}{2} \| q^{n+1} \|^2 - \frac{1}{2}\| \hat{q}^{n+1} \|^2 \leq \delta t\eta \Big(\cL \hat{\Psi}^{n+\frac{1}{2}}, \cN(\overline{\Psi}^{n+\frac{1}{2}}) \cL \hat{\Psi}^{n+\frac{1}{2}}  \Big).
\eeq 
Adding two equations \eqref{eq:proof-cn-step1} and \eqref{eq:proof-cn-step2} together gives us the energy inequality
\begin{multline}
\frac{1}{2} \Big[ (\phi^{n+1}, \cL_0 \phi^{n+1}) + \| q^{n+1}\|^2 \Big] - \frac{1}{2}  \Big[ (\phi^n, \cL_0 \phi^n) + \| q^n\|^2 \Big]  \\
\leq  -\delta t (1-\eta) \Big(\cL \hat{\Psi}^{n+\frac{1}{2}}, \cN(\overline{\Psi}^{n+\frac{1}{2}}) \cL \hat{\Psi}^{n+\frac{1}{2}}  \Big) \leq 0.
\end{multline}
\end{proof}

Similarly, we can propose the relaxed BDF2 scheme as follows.
\begin{scheme}[Relaxed BDF2 Scheme] \label{scheme:BDF2-relax}
After we calculate $\Psi^{n-1}$ and $\Psi^{n}$, we can update $\Psi^{n+1}$ via the following two step schemes:
\begin{itemize}
\item Step1, we obtain the intermediate solutions via the baseline EQ method. Specifically, we calculate 
$\hat{\Psi}^{n+1} := \begin{bmatrix}
\phi^{n+1} \\
\hat{q}^{n+1}
\end{bmatrix}$ via solving the equation
\beq  \label{eq:BDF2-step1}
\frac{3\hat{\Psi}^{n+1} - 4 \Psi^n +  \Psi^n}{2\delta t} = - \cN(\overline{\Psi}^{n+1}) \cL \hat{\Psi}^{n+1}.
\eeq 
\item Step2, we update the baseline solutions via a relaxation step. We update $q^{n+1}$ by  
\beq  \label{eq:BDF2-xi}
q^{n+1} = \xi_0 \hat{q}^{n+1} + (1-\xi_0) h(\phi^{n+1}),
\eeq 
with 
$\xi_0 = \max\{0, \frac{-b - \sqrt{b^2 - 4ac}}{2a}\}$, where the coefficients are
\begin{subequations}
	\begin{align}
	& \nonumber a = \frac{5}{4}   \| \hat{q}^{n+1} - h(\phi^{n+1})\|^2,  b = \frac{1}{2}   \Big( \hat{q}^{n+1}-h(\phi^{n+1}), h(\phi^{n+1}) \Big) + \Big( \hat{q}^{n+1} - h(\phi^{n+1}), 2h(\phi^{n+1}) - q^n \Big), \\
	& \nonumber c = \frac{1}{4} \Big[ \| h(\phi^{n+1})\|^2 + \| 2h(\phi^{n+1}) - q^n \|^2   - \| \hat{q}^{n+1} \|^2 - \| 2\hat{q}^{n+1}-q^n\|^2 \Big]  -  \delta t \eta \Big(\cL \hat{\Psi}^{n+1}, \cN(\overline{\Psi}^{n+1}) \cL \hat{\Psi}^{n+1}  \Big).
	\end{align}
	Here $\eta \in [0, 1]$ is an artificial parameter that can be manually assigned. 
\end{subequations}
\end{itemize}
\end{scheme}

\begin{rem}
In a similar manner, the readers shall notice the extra relaxation step \eqref{eq:BDF2-xi}  is cheap-and-easy to compute.
We comment on the idea for the relaxation step. $\xi$ is a solution of the optimization problem
\begin{multline}
\xi_0  = \min_{\xi \in [0, 1]} \xi, \quad \mbox{ subject to } \quad \frac{1}{4}( \| q^{n+1}\|^2 + \| 2q^{n+1} - q^n\|^2) - \\
\frac{1}{4}( \| \hat{q}^{n+1} \|^2 + \| 2\hat{q}^{n+1}-q^n\|^2) \leq \delta t \eta \Big(\cL \hat{\Psi}^{n+1}, \cN(\overline{\Psi}^{n+1}) \cL \hat{\Psi}^{n+1}  \Big),
\end{multline}
where $\eta \in [0, 1]$ is an artificial parameter that can be manually assigned. 
\end{rem}

\begin{rem}
 In practice, it can be simplified as solving the following algebra optimization problem
\beq
\xi_0 = \arg\min_{\xi \in [0, 1]} \xi,  \quad \mbox{  subject to } a \xi^2 + b\xi + c \leq 0,
\eeq 
where the coefficients for the quadratic equation are
\begin{subequations}
\begin{align}
& \nonumber a = \frac{5}{4}   \| \hat{q}^{n+1} - h(\phi^{n+1})\|^2,  b = \frac{1}{2}   \Big( \hat{q}^{n+1}-h(\phi^{n+1}), h(\phi^{n+1}) \Big) + \Big( \hat{q}^{n+1} - h(\phi^{n+1}), 2h(\phi^{n+1}) - q^n \Big) \\
& \nonumber c = \frac{1}{4} \Big[ \| h(\phi^{n+1})\|^2 + \| 2h(\phi^{n+1}) - q^n \|^2   - \| \hat{q}^{n+1} \|^2 - \| 2\hat{q}^{n+1}-q^n\|^2 \Big]  -  \delta t \eta \Big(\cL \hat{\Psi}^{n+1}, \cN(\overline{\Psi}^{n+1}) \cL \hat{\Psi}^{n+1}  \Big).
\end{align}
\end{subequations}
By a similar argument, we have 
$$\xi_0 = \max\{0, \frac{-b - \sqrt{b^2 - 4ac}}{2a}\}.$$
\end{rem}

\begin{theorem}
The Scheme \ref{scheme:BDF2-relax} is unconditionally energy stable.
\end{theorem}

\begin{proof}
The proof is similar with the proof for Scheme \ref{scheme:CN-relax}. In fact, if we take inner product of \eqref{eq:BDF2-step1} with $\delta t \hat{\Psi}^{n+1}$, we will have
\begin{multline} \label{eq:proof-bdf2-1}
\frac{1}{4} \Big[ (\phi^{n+1}, \cL_0 \phi^{n+1}) + (2\phi^{n+1} - \phi^n, \cL_0(2\phi^{n+1} -\phi^n)) + \| \hat{q}^{n+1} \|^2 + \| 2 \hat{q}^{n+1} - q^n\|^2 \Big] \\
- \frac{1}{4} \Big[ (\phi^n, \cL_0 \phi^n) + (2\phi^n- \phi^{n-1}, \cL_0(2\phi^n -\phi^{n-1})) + \| q^n \|^2 + \| 2 q^n- q^{n-1}\|^2 \Big] \\
\leq - \delta t  \Big( \cL \hat{\Psi}^{n+1}, \cN(\overline{\Psi}^{n+1}) \cL \hat{\Psi}^{n+1}  \Big).
\end{multline}
Meanwhile, we know from \eqref{eq:BDF2-xi} that
\beq \label{eq:proof-bdf2-2}
\frac{1}{4}( \| q^{n+1}\|^2 + \| 2q^{n+1} - q^n\|^2) - 
\frac{1}{4}( \| \hat{q}^{n+1} \|^2 + \| 2\hat{q}^{n+1}-q^n\|^2) \leq \delta t \eta \Big(\cL \hat{\Psi}^{n+1}, \cN(\overline{\Psi}^{n+1}) \cL \hat{\Psi}^{n+1}  \Big).
\eeq 
Adding the equations \eqref{eq:proof-bdf2-1} and \eqref{eq:proof-bdf2-2} together, we will have
\begin{multline} 
\frac{1}{4} \Big[ (\phi^{n+1}, \cL_0 \phi^{n+1}) + (2\phi^{n+1} - \phi^n, \cL_0(2\phi^{n+1} -\phi^n)) + \| q^{n+1} \|^2 + \| 2 q^{n+1} - q^n\|^2 \Big] \\
- \frac{1}{4} \Big[ (\phi^n, \cL_0 \phi^n) + (2\phi^n- \phi^{n-1}, \cL_0(2\phi^n -\phi^{n-1})) + \| q^n \|^2 + \| 2 q^n- q^{n-1}\|^2 \Big] \\
\leq - \delta t (1-\eta) \Big(\cL \hat{\Psi}^{n+1}, \cN(\overline{\Psi}^{n+1}) \cL \hat{\Psi}^{n+1}  \Big) \leq 0,
\end{multline}
since $1-\eta \leq 0$.
This completes the proof.

\end{proof}

\section{Numerical Results}
Next, we test this novel relaxed-EQ method on several specific thermodynamically consistent models. We consider problems with periodic boundary conditions for simplicity, though we emphasize our approach has no specific restriction on the boundary conditions. For spatial discretization, we use the Fourier Pseudo-spectral method. Due to space limitation, we only test the CN schemes, i.e. Scheme \ref{scheme:CN} and Scheme \ref{scheme:CN-relax} and choose $\eta = 1$.

For the AC equation and the CH equation, we consider a square domain $\Omega=[0, 1]^2$ with seven disks in the 2D domain as an initial condition, similar with the problem in \cite{GuoBenchmark}. It is known that the disks will shrink and eventually disappear for the dynamics driven by the AC equation since the AC equation doesn't preserve the total volume. Meanwhile,  the Ostwald ripening effect will occur for the CH equation dynamics since the CH equation maintains the total volume while minimizing the interfaces. Numerical results for the two dynamics are summarized in Figure \ref{fig:AC-CH}(a) and \ref{fig:AC-CH}(b), respectively, where we observe the expected dynamics. 

\begin{figure}[H]
\subfigure[dynamics driven by the AC equation with $\phi$  at $t=0,10,50,60$]{
\includegraphics[width=0.24\textwidth]{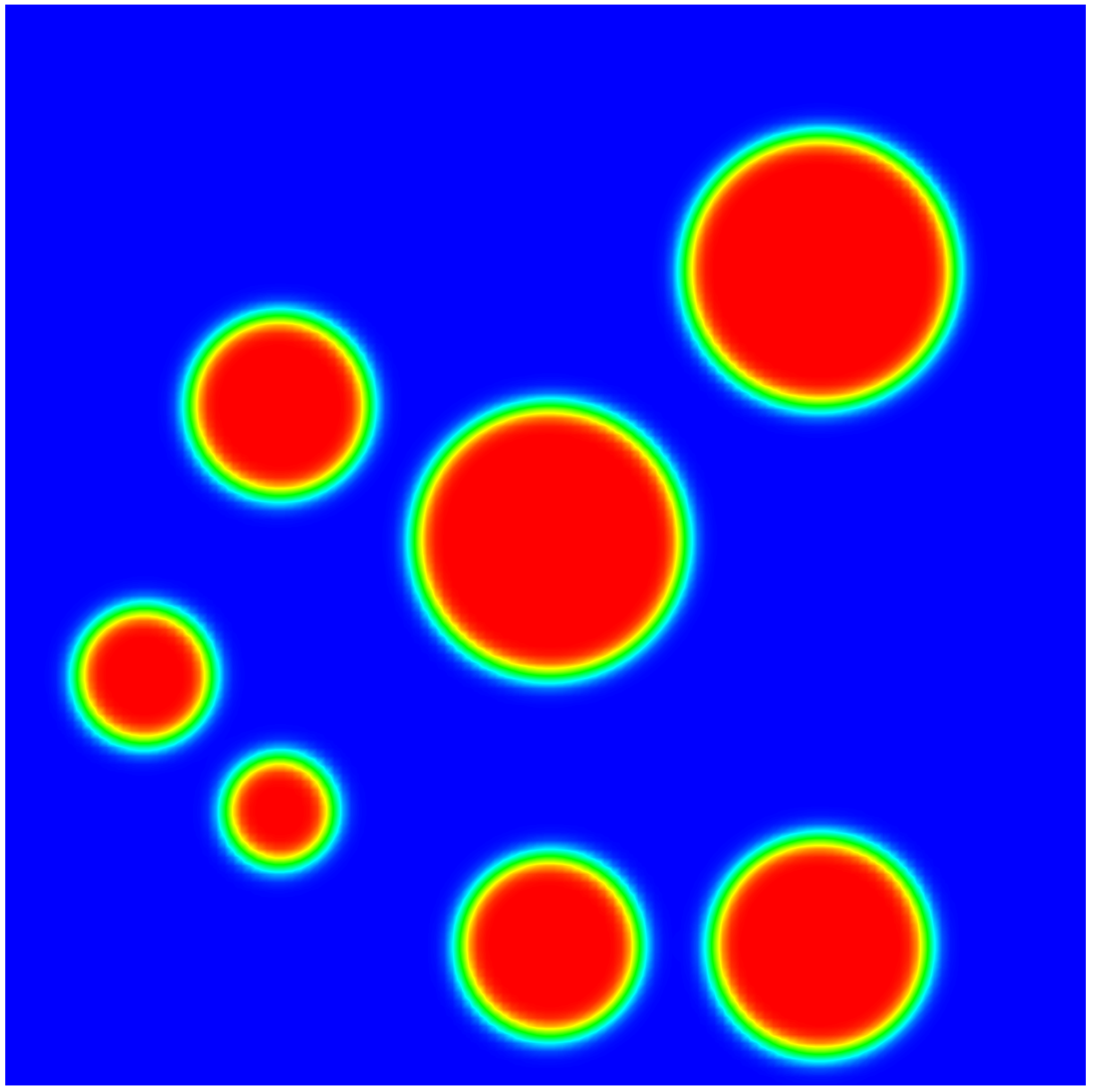}
\includegraphics[width=0.24\textwidth]{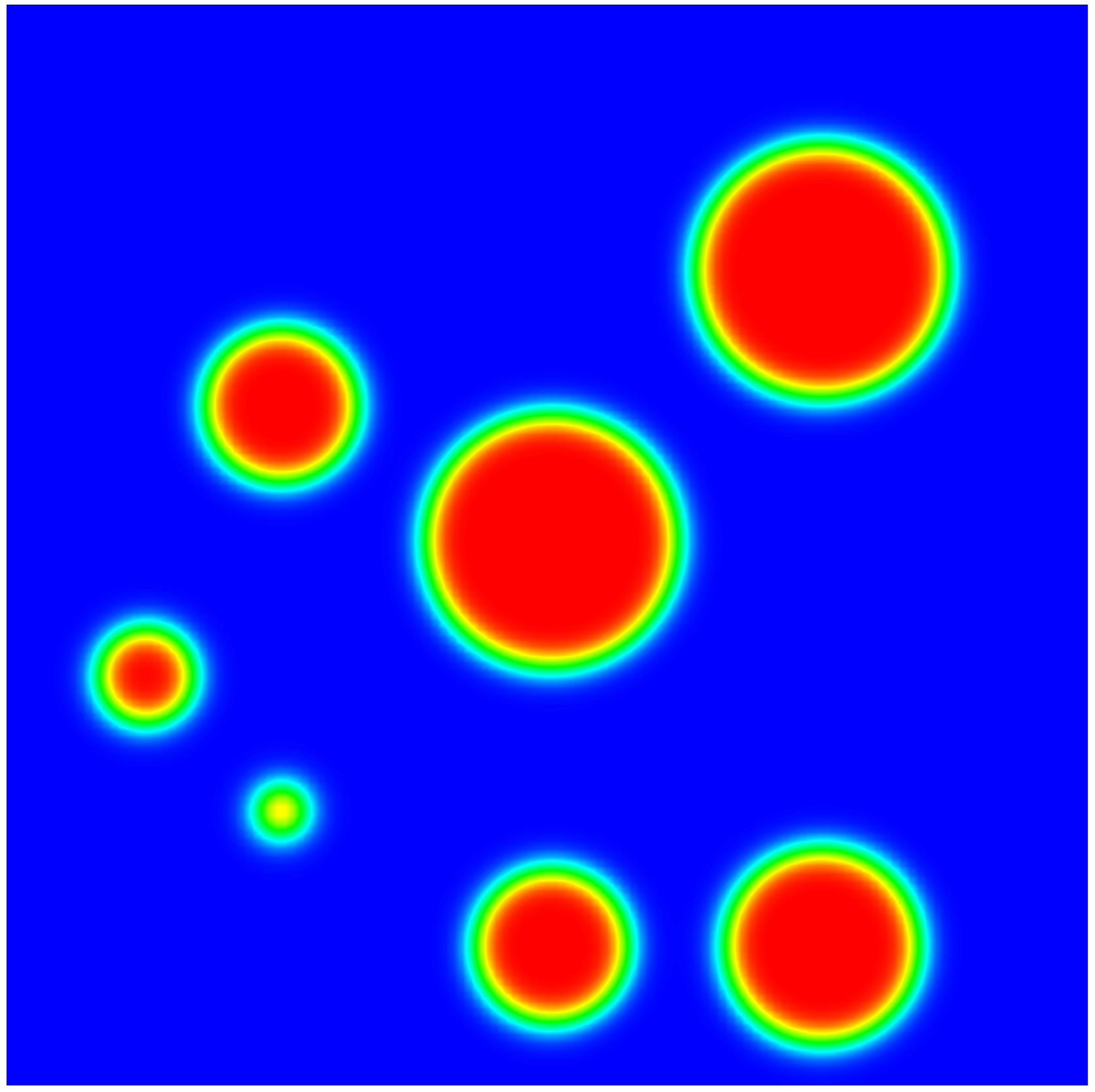}
\includegraphics[width=0.24\textwidth]{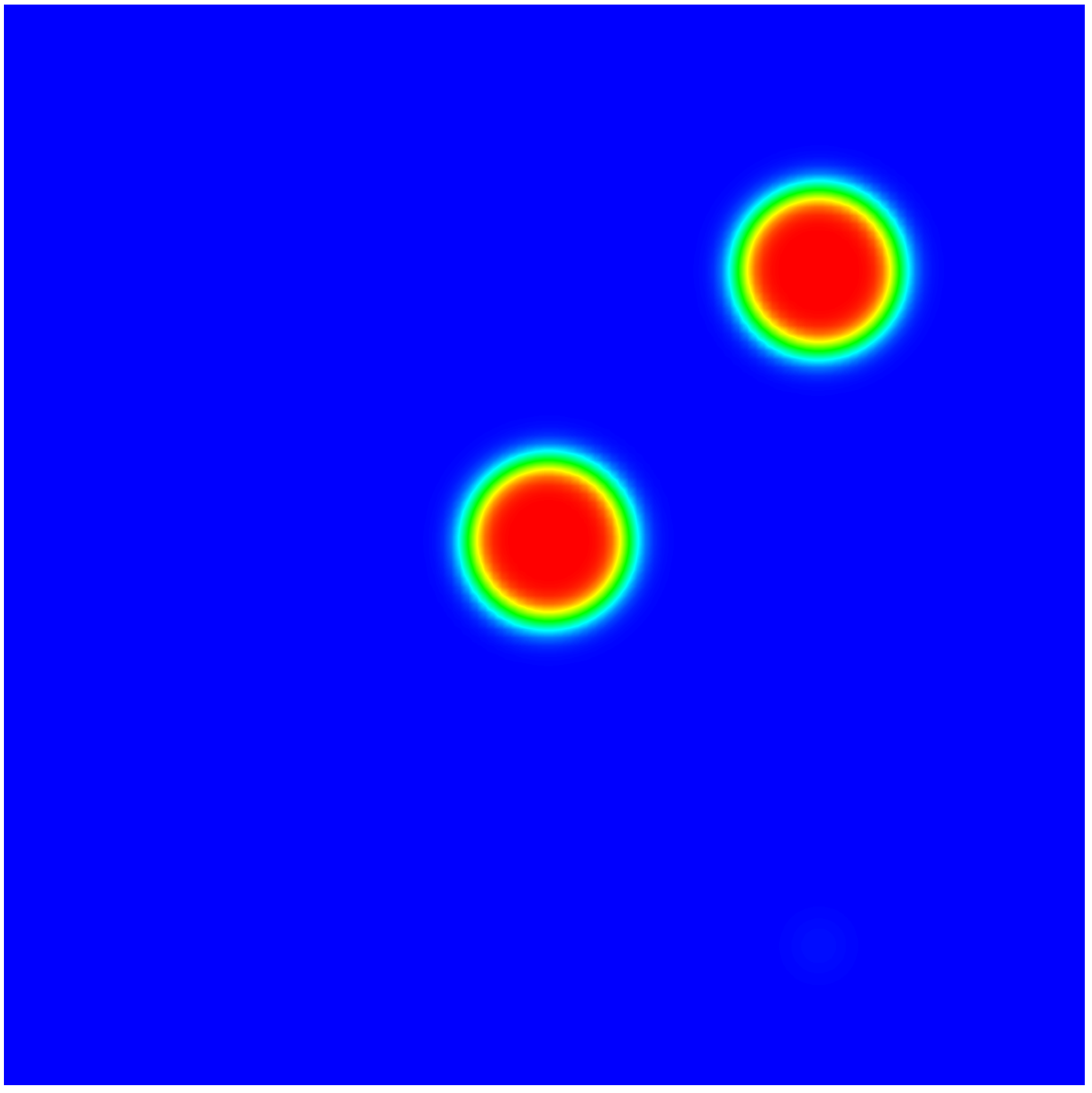}
\includegraphics[width=0.24\textwidth]{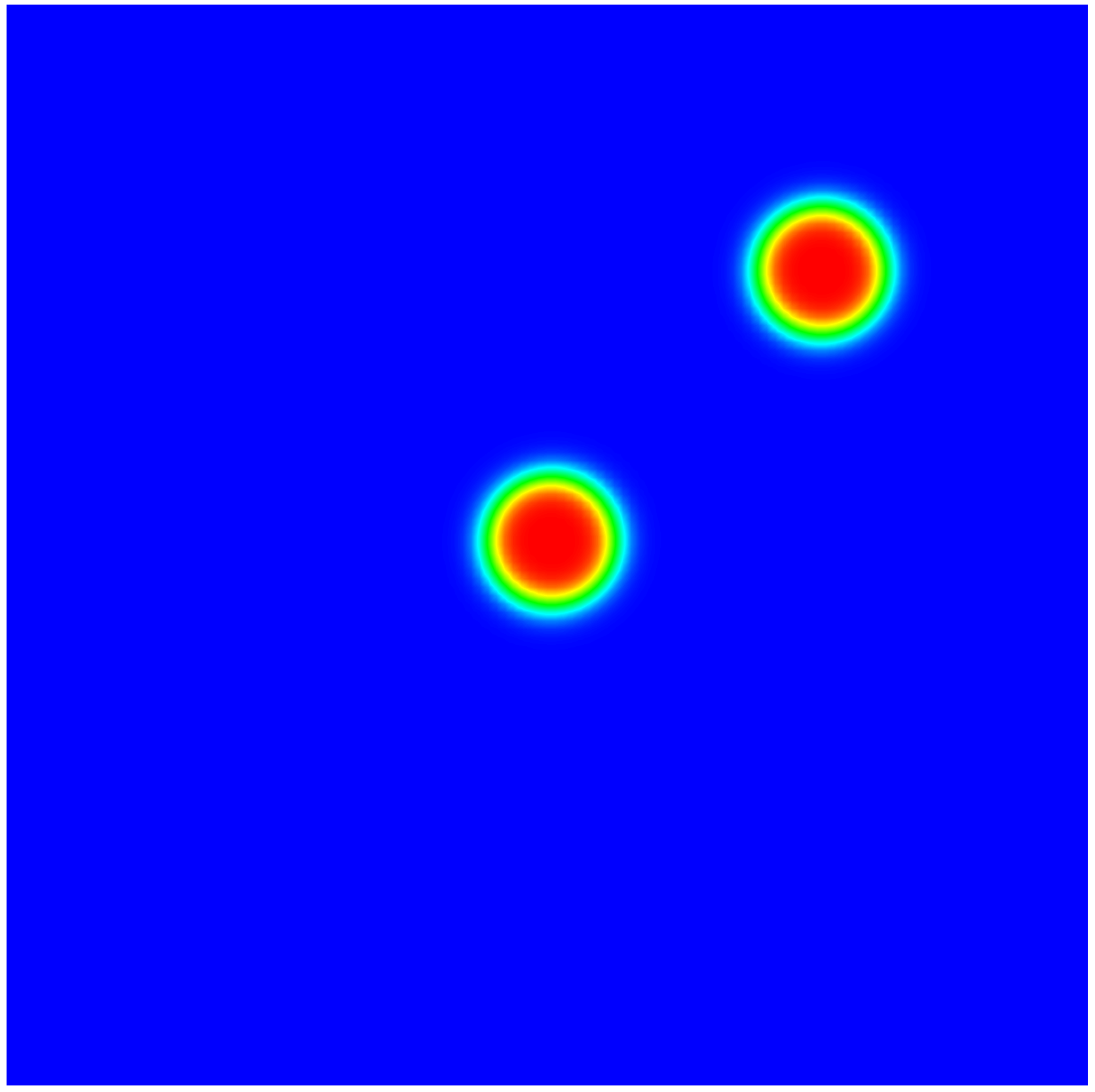}
}

\subfigure[dynamics driven by the CH equation with $\phi$ at $t=0,10,50,100$]{
\includegraphics[width=0.24\textwidth]{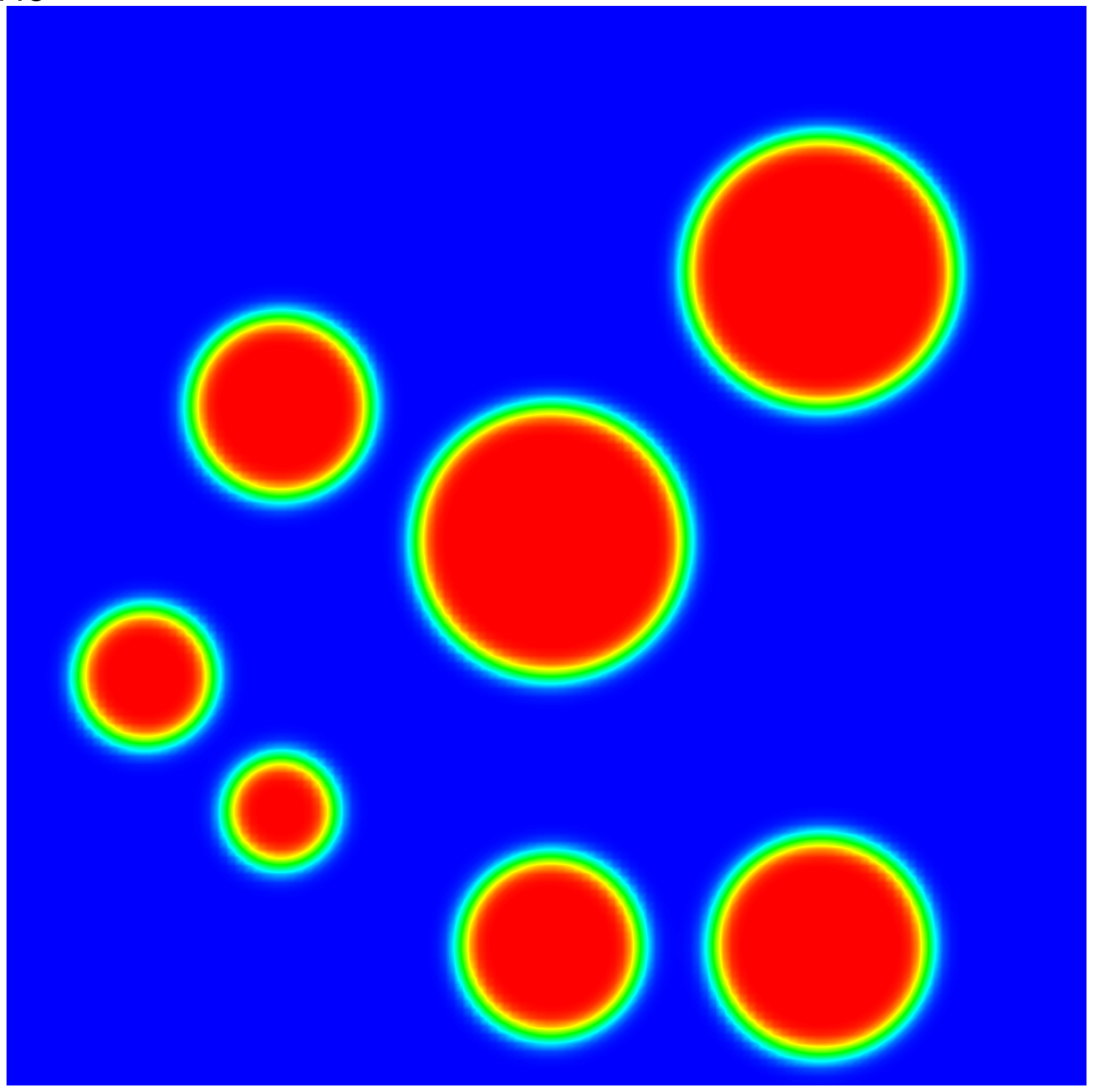}
\includegraphics[width=0.24\textwidth]{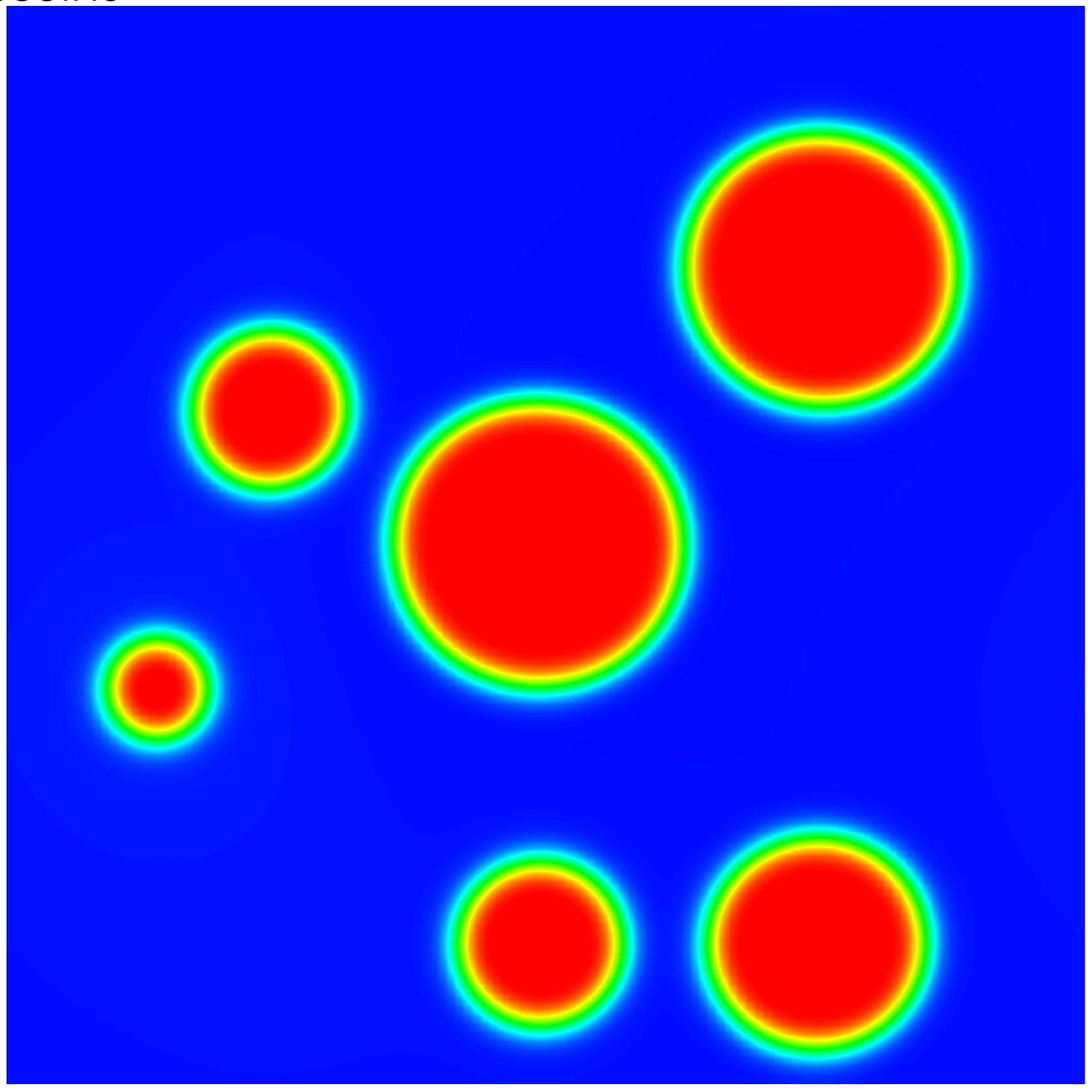}
\includegraphics[width=0.24\textwidth]{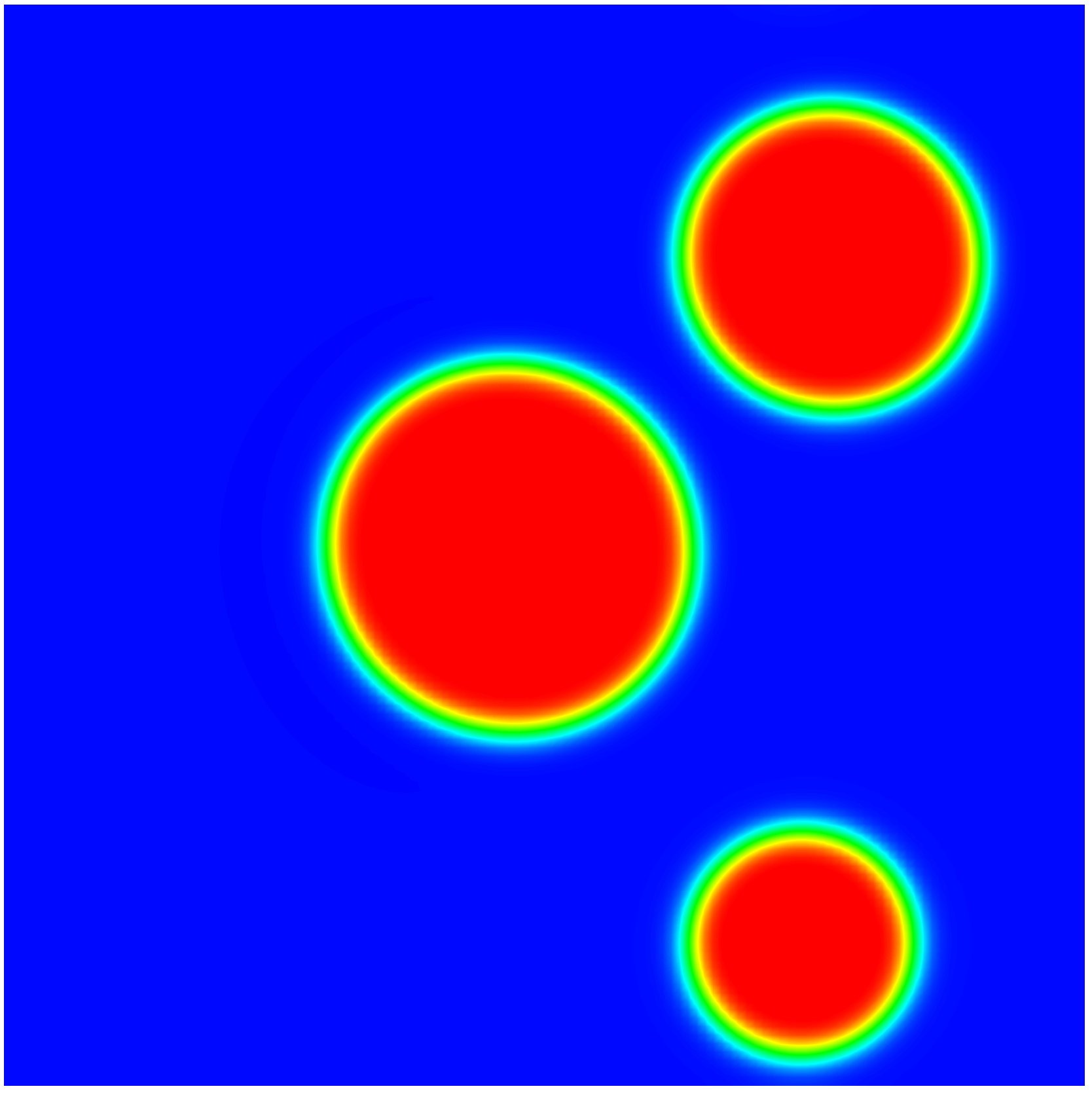}
\includegraphics[width=0.24\textwidth]{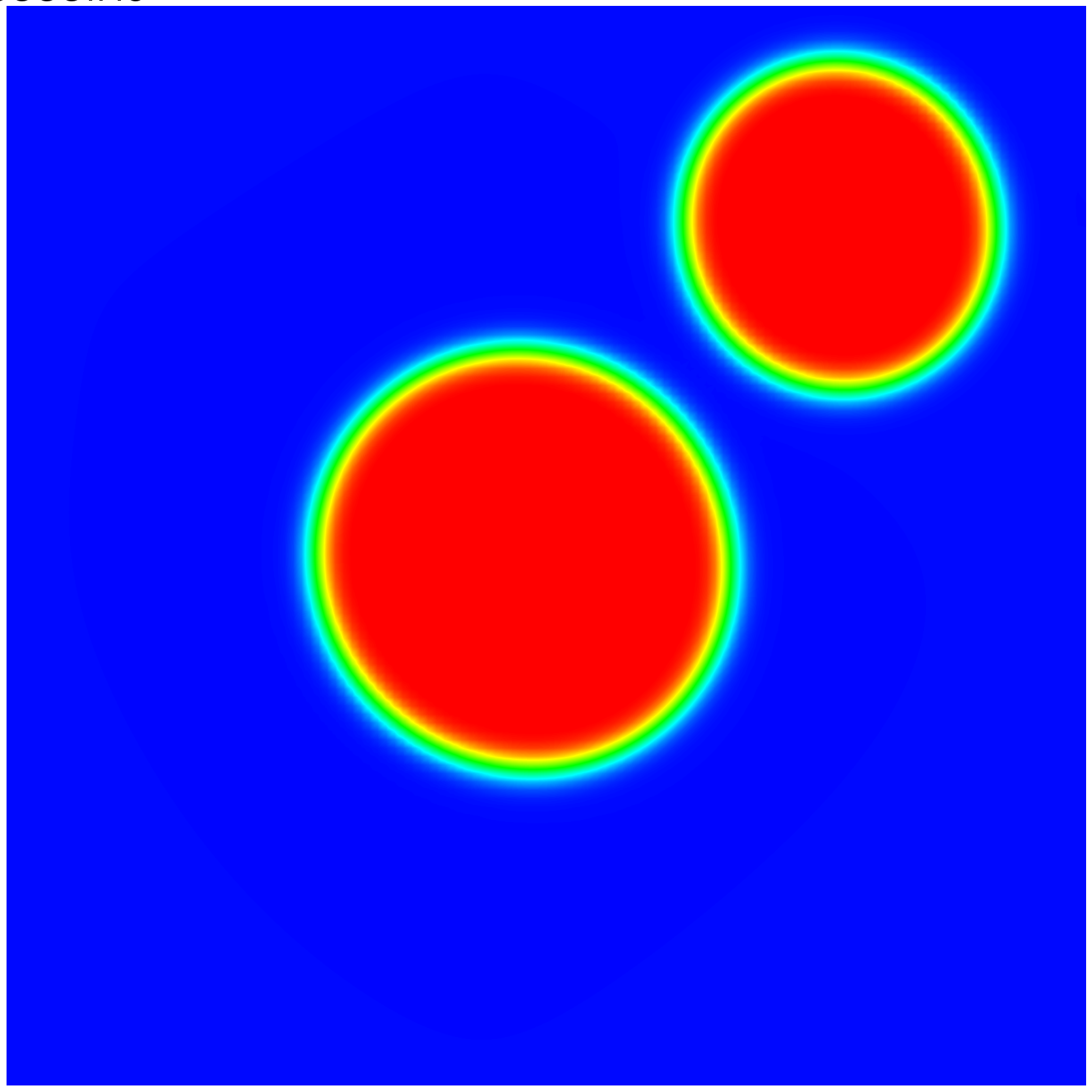}
}
\caption{Evolution dynamics driven by the Allen-Cahn equation and Cahn-Hilliard equation respectively. (a) the profiles of $\phi$ at various time slots driven by the AC equation; (b) the profiles of $\phi$ at various time slots driven by the CH equation.}
\label{fig:AC-CH}
\end{figure}

One particular focus of this letter is how the relaxed-EQ method performs compared with the baseline EQ method.  We solve the problem in Figure \ref{fig:AC-CH} using both the EQ and relaxed-EQ methods with $128^2$ meshes and a large time step, namely, $\delta t = 0.75$ for the AC equation, and $\delta t=0.005$ for the CH equation. The numerical energies using both approaches are summarized in Figure \ref{fig:Energy}(a) and \ref{fig:Energy}(b) respectively for the AC and CH equations. Notice both algorithms perform well when the time step is small enough. When the time step is large, the relaxed-EQ method outperforms the baseline EQ method.

\begin{figure}
\center
\subfigure[Energy with the AC equation]{\includegraphics[width=0.32\textwidth]{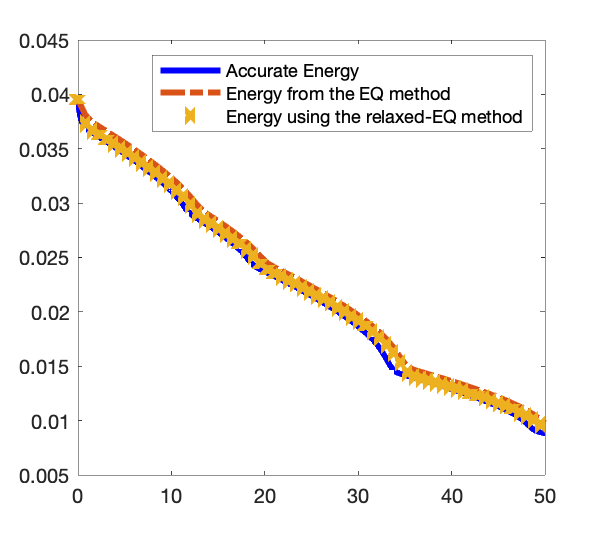}}
\subfigure[Energy with the CH equation]{\includegraphics[width=0.32\textwidth]{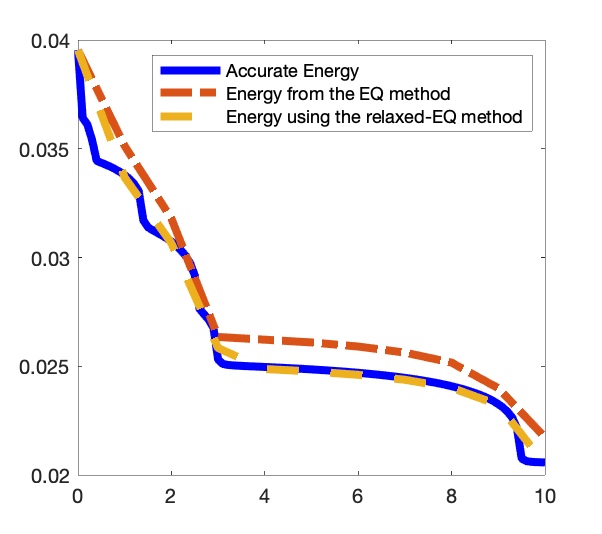}}
\subfigure[Energy with the MBE equation]{\includegraphics[width=0.32\textwidth]{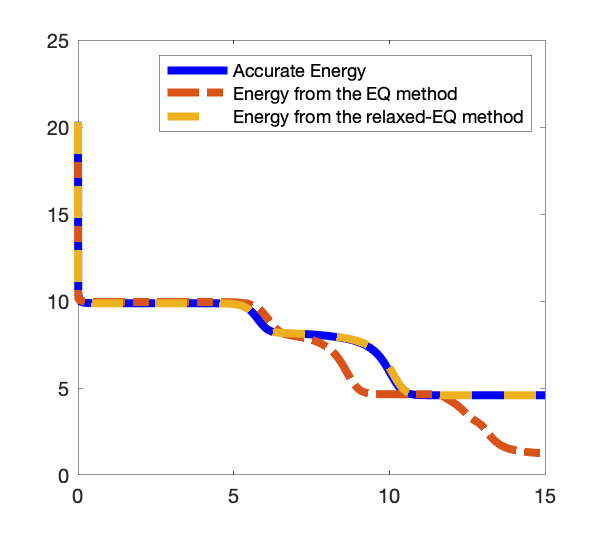}}
\caption{Numerical energy comparisons between the EQ method and the relaxed-EQ method. In this figure, we test the AC equation, the CH equation and the MBE equation with both numerical methods using the same time step. This figure indicates that the relaxed-EQ method provides more accurate results for all cases than the baseline EQ method with a same time step.}
\label{fig:Energy}
\end{figure}

Furthermore, we also use both methods to solve the MBE equation. We use classical benchmark problems from \cite{Wang&Wang&WiseDCDS2010, Li&Liu2003},  with the same initial conditions and parameters. Due to space limitations, the numerical results are not shown. Instead, the energy evolutions using both numerical methods with a time step $\delta t = 0.001$ are summarized in Figure \ref{fig:Energy}(c). We observe that the relaxed-EQ method provides more accurate results than the EQ method for solving the MBE model as well.

In the last example, we solve the PFC model using the relaxed-EQ scheme to further highlight its power for simulating long-time dynamics. We set up the initial conditions as Figure \ref{fig:PFC}(a) with three pieces of crystal blocks. The numerical results of crystal growth dynamics are summarized in Figure \ref{fig:PFC}.  We observe dynamics with the appearance of micro-structure and dislocations, similarly as reported in the literature.

\begin{figure}[H]
\center
\subfigure[the profile of $\phi$ at $t=0,50,75$]{
\includegraphics[width=0.3\textwidth]{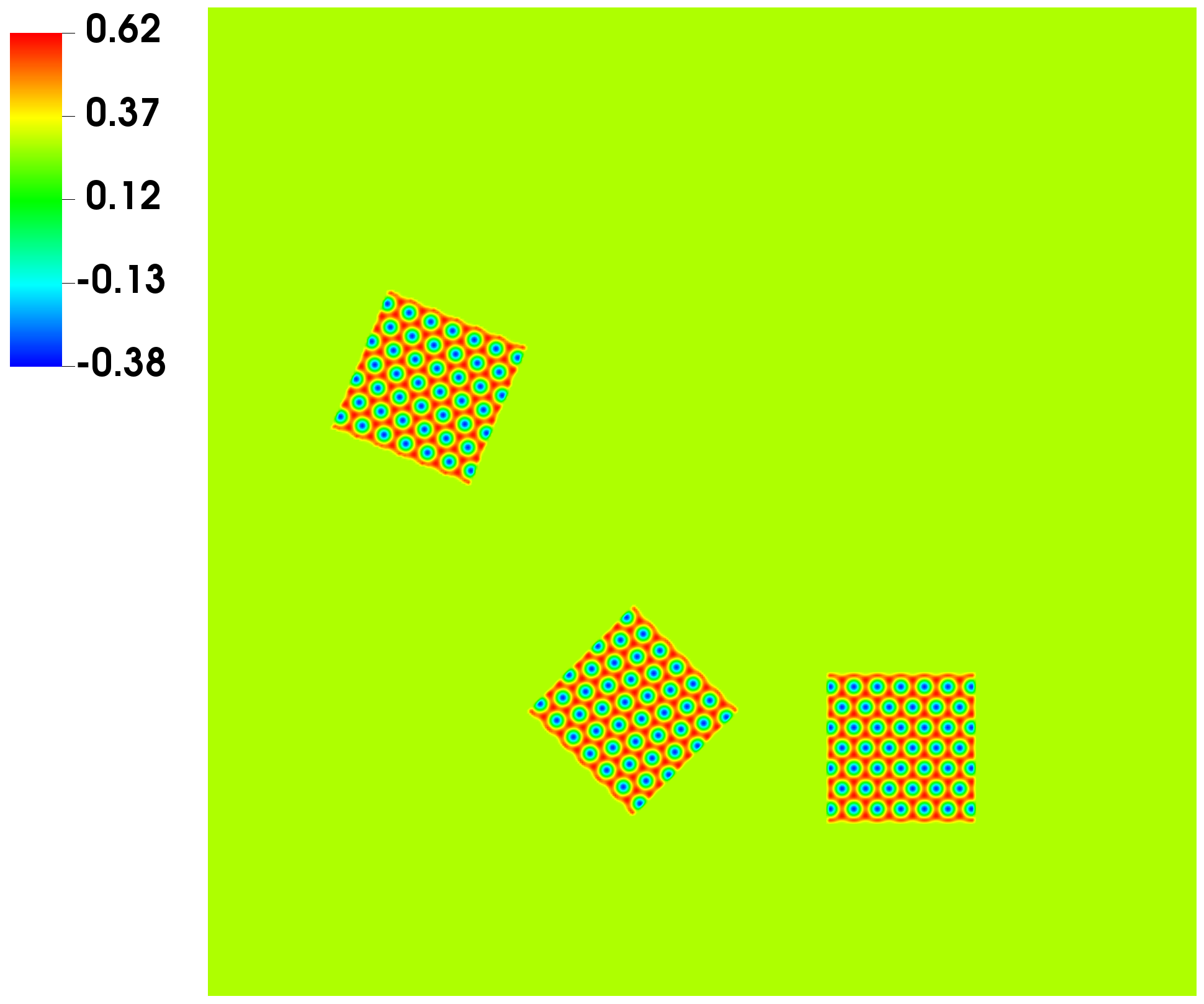}
\includegraphics[width=0.3\textwidth]{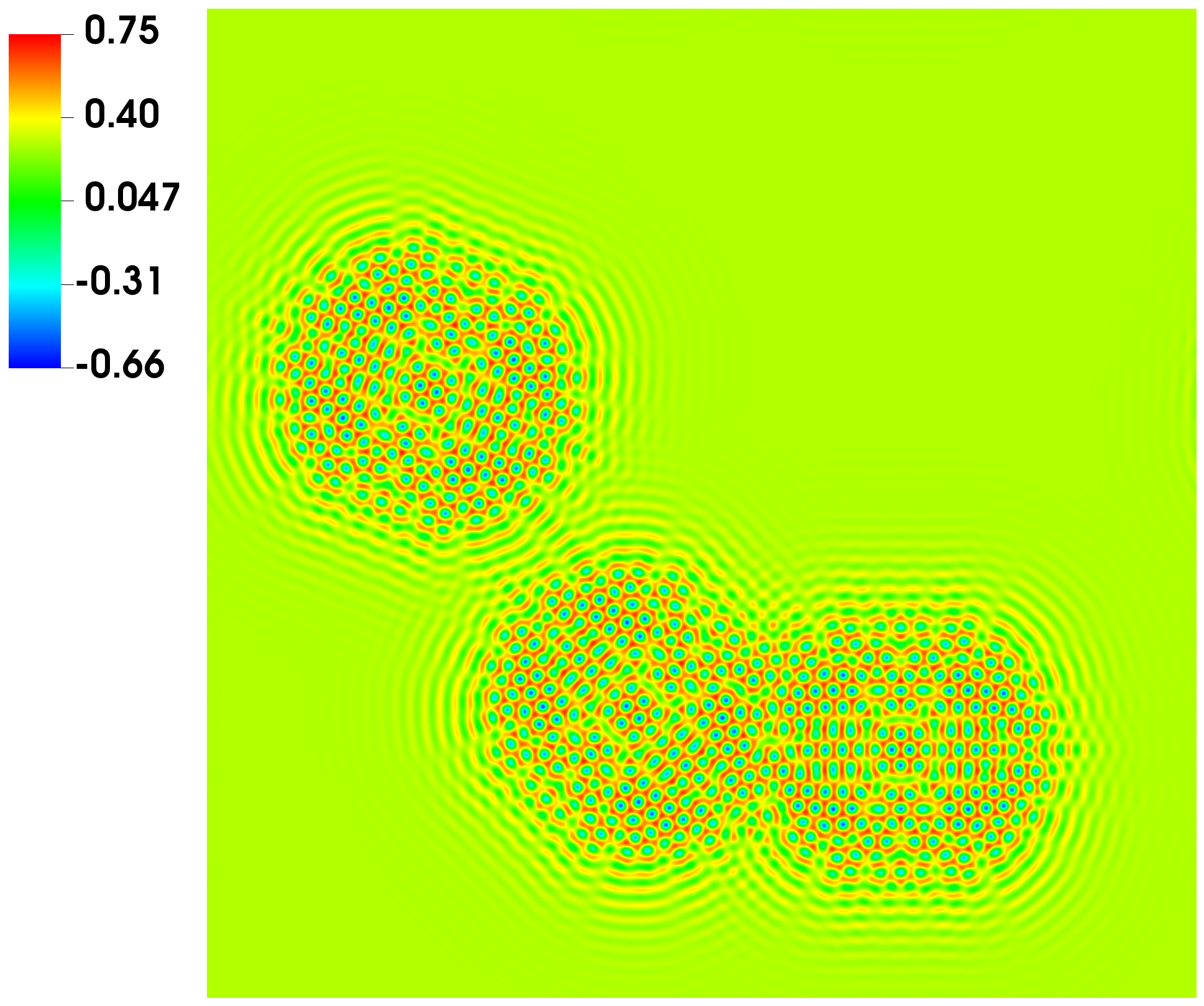}
\includegraphics[width=0.3\textwidth]{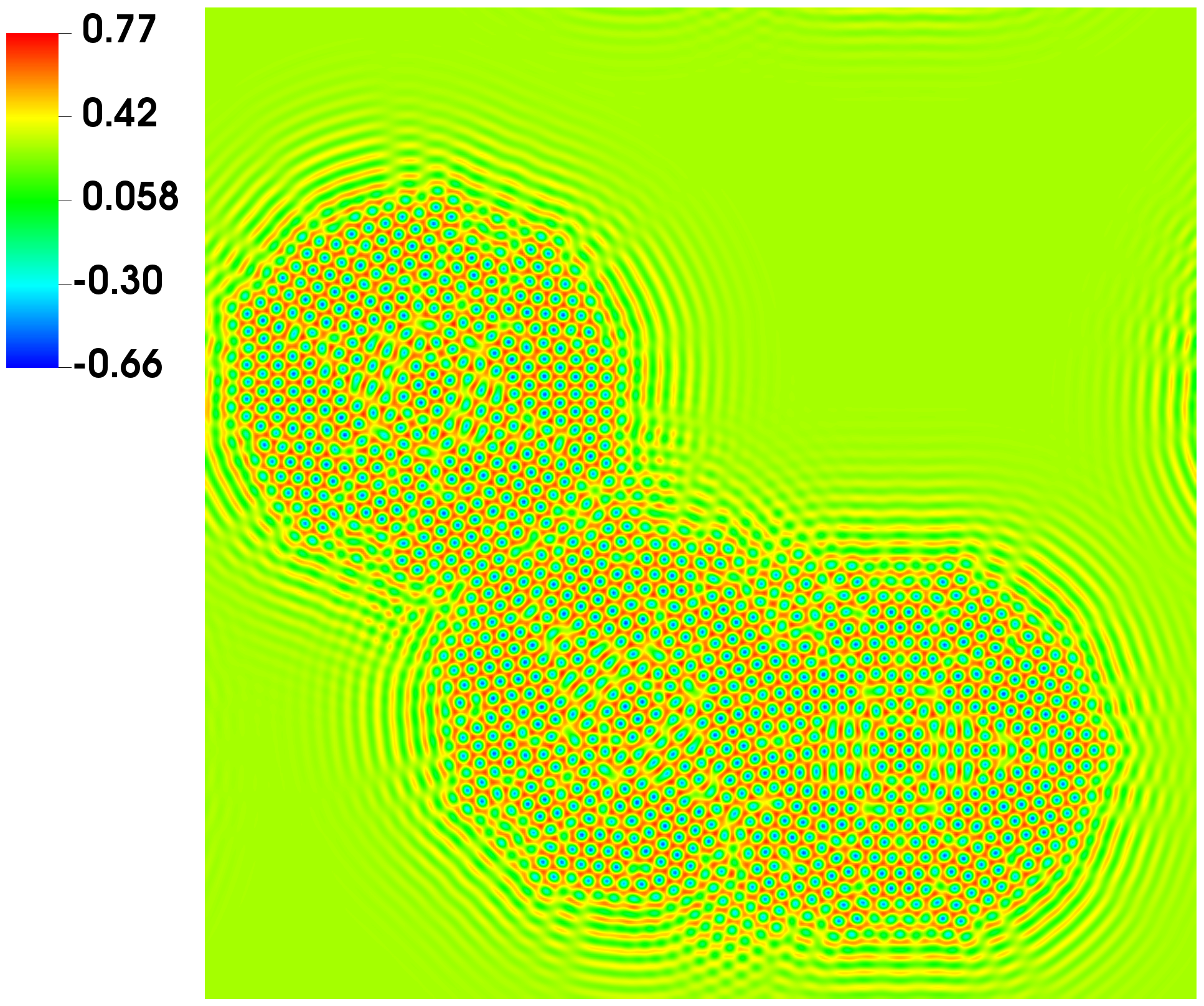}
}

\subfigure[the profile of $\phi$ at $t=100,150,1800$]{
\includegraphics[width=0.3\textwidth]{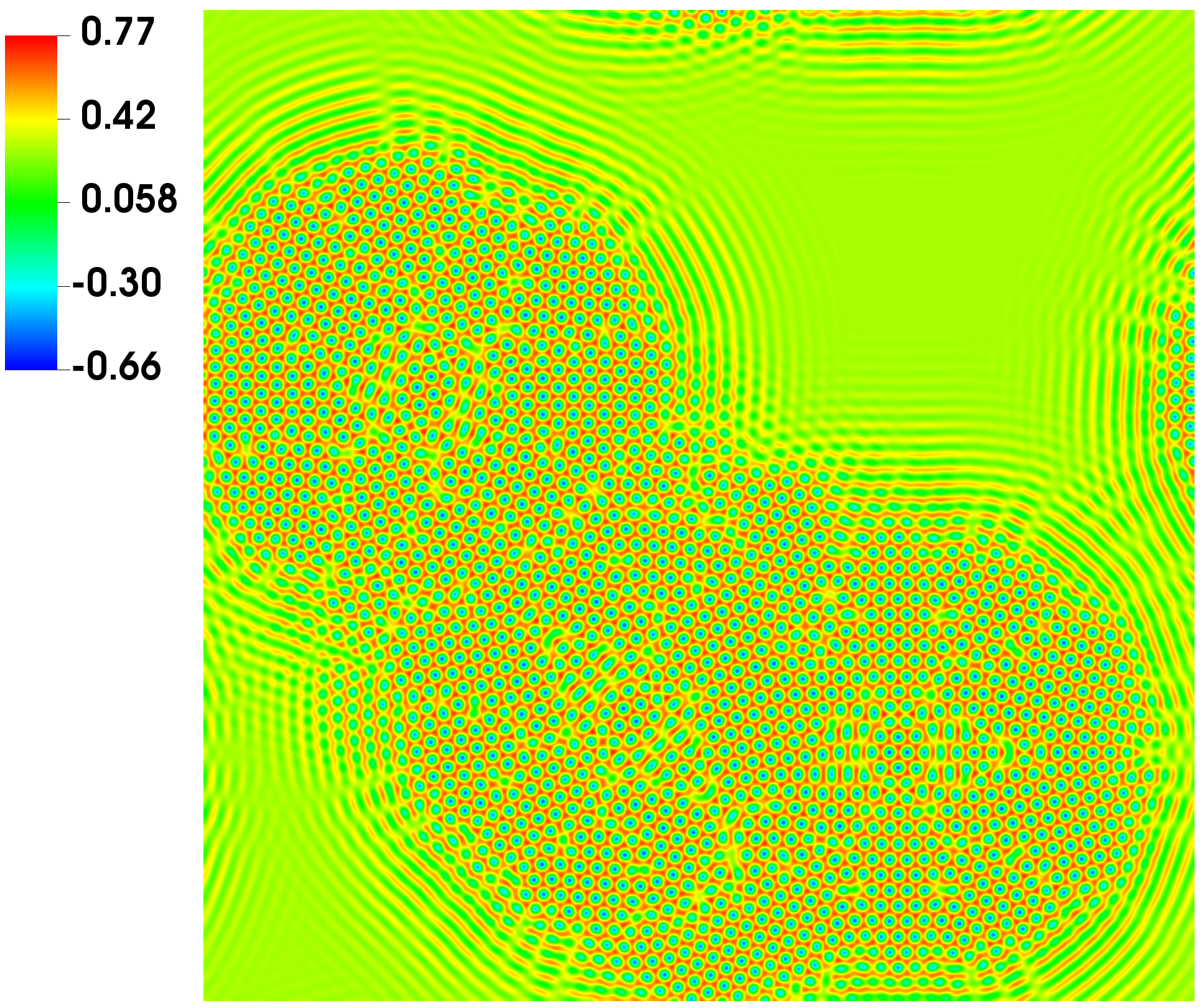}
\includegraphics[width=0.3\textwidth]{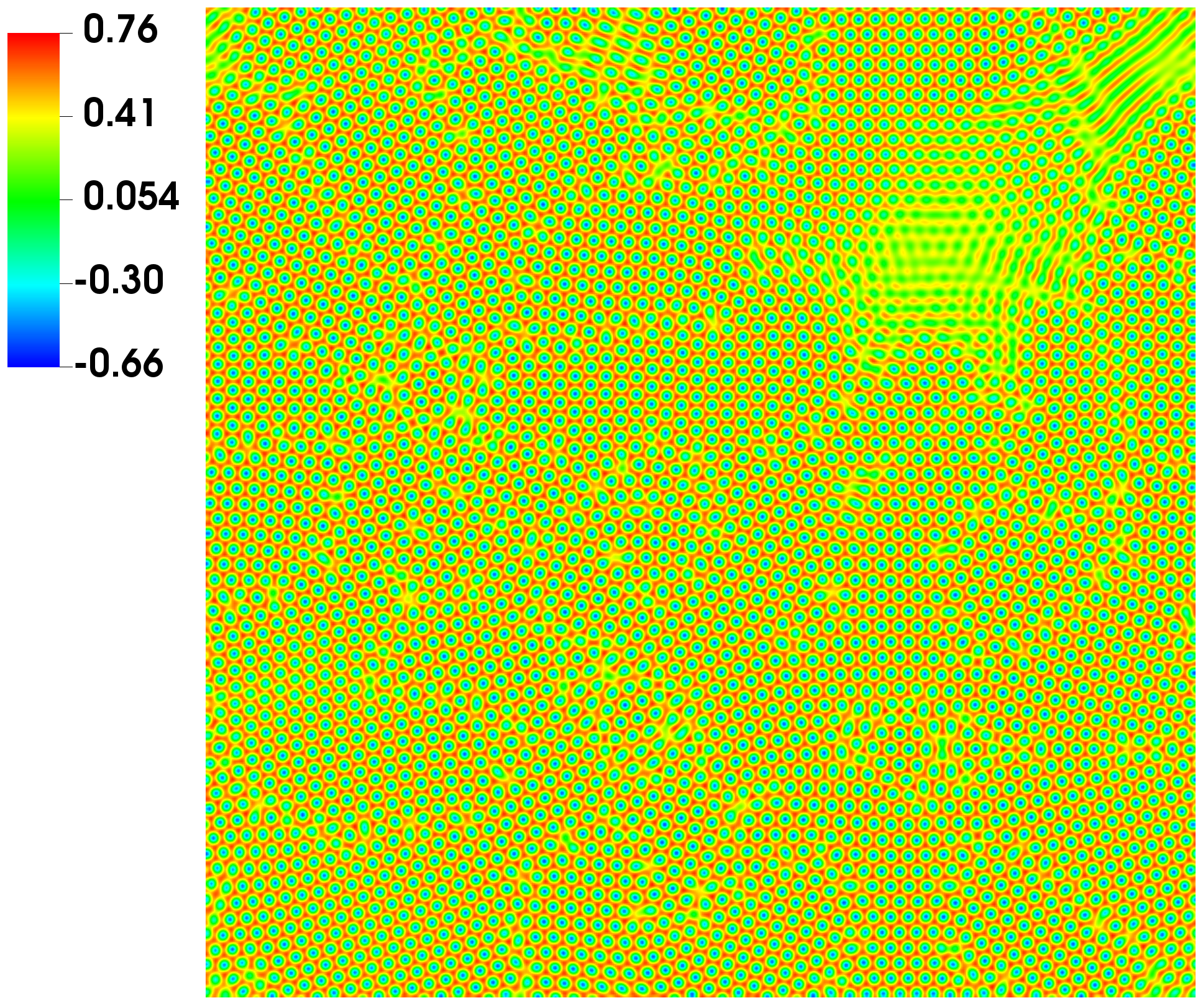}
\includegraphics[width=0.3\textwidth]{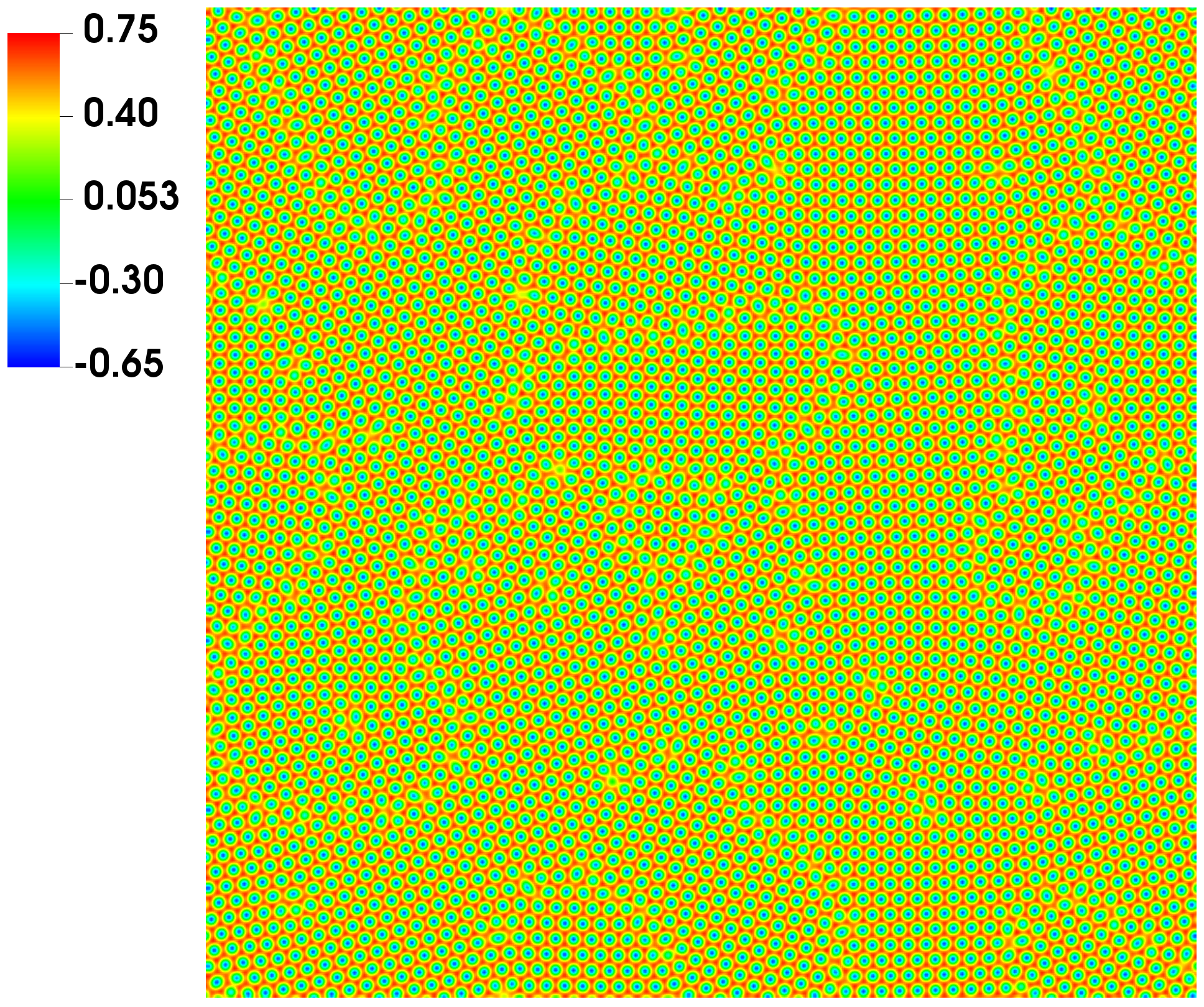}
}
\caption{Microstructure evolution dynamics for triangular crystal phases driven by the phase field crystal model.  The profiles of $\phi$ at various time slots are shown. }
\label{fig:PFC}
\end{figure}

\section{Conclusion}
This letter reports a relaxation technique to overcome the inconsistency issue of the modified free energy and the original energy in the energy quadratization (EQ) method. Through theorems and numerical examples, it is noticeable that the relaxed-EQ method always outperforms the baseline EQ method. In particular, it overcomes the EQ method's primary issue that the auxiliary variable deviates from its original continuous definition after discretization as numerical errors are introduced and accumulated. In the relaxed-EQ method, the numerical solution of the auxiliary variable is relaxed towards its continuous definition in each time step with negligible computational cost. Overall, this letter's main message is that the relaxation step shall be applied to all available EQ schemes in literature with its easy-to-use and accuracy-improving nature.

\section*{Acknowledgments}
Jia Zhao would like to thank Prof. Qi Wang from the University of South Carolina for inspiring discussions on the generalized Onsager principles. Jia Zhao would also like to acknowledge the support from National Science Foundation with grant NSF-DMS-1816783, and NVIDIA Corporation for the donation of a Quadro P6000 GPU for conducting some of the numerical simulations in this paper.

\bibliographystyle{unsrt}

\end{document}